\documentclass[11pt]{amsart}

\usepackage{amsmath}
\usepackage{amssymb}
\usepackage{amscd}

\usepackage{enumitem}
\usepackage{mathtools}

\usepackage{tikz-cd}
\usepackage{bm}
\usepackage{ytableau}

\usepackage{hyperref}
\usepackage{biblatex}
\DeclareMathOperator{\rk}{rk}

\newtheorem*{conjecture}{Conjecture}
\newtheorem*{proposition}{Proposition}

\theoremstyle{remark}
\newtheorem*{remark}{Remark}

\title{Derived category of moduli of parabolic bundles on $\mathbb{P}^1$}
\author{Anton Fonarev}
\dedicatory{Dedicated to the memory of M.\,S.\,Narasimhan.}
\address{\sloppy
\parbox{0.95\textwidth}{
Algebraic Geometry Section, Steklov Mathematical Institute of Russian Academy of Sciences,
8 Gubkin str., Moscow 119991 Russia\\[.5em]
National Research University Higher School of Economics
6 Usacheva str., Moscow 119048 Russia
\hfill
}\bigskip}
\email{avfonarev@mi-ras.ru}
\date{}
\thanks{This work was supported by the Russian Science Foundation under grant no.~21-11-00153, https://rscf.ru/en/project/21-11-00153/}

\begin{document}

\begin{abstract}
    We propose a conjecture on the structure of the bounded
    derived category of coherent sheaves of the moduli
    space rank $2$ parabolic bundles on $\mathbb{P}^1$.
\end{abstract}

\maketitle

\section{Moduli of bundles on curves}
For simplicity, we work over the field of complex numbers.
Let $C$ be a smooth projective curve of genus $g \geq 1$. In~\cite{FK}
it was shown that for a general $C$
its bounded derived category of coherent sheaves $D(C)$
embeds in the derived category $D(M)$,
where $M$ is the moduli space of stable rank $2$ bundles on~$C$
with fixed determinant of odd degree.
This result was independently obtained by Narasimhan in~\cite{Na} for any curve
of genus $g\geq 4$ without the generality assumption. These were the first
steps towards a~conjecture, which is commonly attributed to Narasimhan and, independently, to Belmans, Galkin, and Mukhopadhyay.
After plenty of work done by various groups of authors,
this conjecture was seemingly proved in the preprint~\cite{Te}.
Namely, it was shown that there is a semiorthogonal decomposition
\begin{equation}\label{eq:dbcurves}
    D(M) = \langle \mathcal{O}, \mathcal{O}(1),\ D(C), D(C)(1),\ \ldots,
    D(S^{g-2}C), D(S^{g-2}C)(1),\ D(S^{g-1}C) \rangle,
\end{equation}
where $S^iC$ denotes the $i$-th symmetric power of the curve $C$.

A key ingredient in~\cite{FK} was an explicit geometric description of
$M$ for hyperelliptic curves. Consider a~hyperelliptic curve $C$.
Pick a coordiate on $\mathbb{P}^1$ so that the branching points
$p_i=(1:a_i)$, $i=1,\ldots,2g+2$ of the hyperelliptic projection are not
at infinity.
With the curve $C$ one associates a net of quadrics generated by
\begin{equation}\label{eq:net}
    q_0=a_1x_1^2+a_2x_2^2+\cdots+a_{d}x_{d}^2, \quad
    q_\infty = -(x_1^2+x_2^2+\cdots+x_{d}^2),
\end{equation}
where $d=2g+2$, and $x_i$ for $i=1,\ldots,d$ are coordinates on a fixed 
$d$-dimensional vector space $V$.
A~celebrated theorem of Desale and Ramanan states that
$M$ is isomorphic to the space of
$(g-1)$-dimensional subspaces in~$V$ isotropic with respect to $q_0$ and $q_\infty$
(see~\cite{DR}). Equivalently, this is the variety of projective subspaces
$\mathbb{P}^{g-2}$ lying on the intersection of two even-dimensional quadrics
$Q_0$ and $Q_\infty$ defined by $q_0$ and $q_\infty$, respectively.
It seems natural to ask what happens when the the quadrics are odd-dimensional.

\section{Moduli of parabolic bundles on \texorpdfstring{$\mathbb{P}^1$}{P1}}
Let $V$ be a vector space of dimension $d=2g+1$ for some $g>1$.
Consider the~net of quadrics~\eqref{eq:net},
for distinct
$a_1,\ldots,a_{2g+1}$.
It turns out that the variety of subspaces of~dimension $g-1$ in $V$
isotropic with respect to $q_0$ and $q_\infty$
is isomorphic to the moduli space of stable quasiparabolic bundles on rank $2$
and dergree $0$ on $\mathbb{P}^1$ with weights $\frac{1}{2}$ at the marked points
$p_i=(1:a_i)$ (see~\cite{Ca}).
Meanwhile, this moduli space is isomorphic ot the moduli space $\mathcal{M}$ of rank $2$
bundles on stacky $\mathbb{P}^1$ with $\mathbb{Z}/2\mathbb{Z}$
structure at the points $p_i$ (see~\cite{Bi}).
Denote the latter stack by $\mathcal{C}$.
The following conjecture generalises decomposition~\eqref{eq:dbcurves}.

\begin{conjecture}
    There is a semiorthogonal decomposition
    \begin{equation}\label{eq:dbstack}
        D(\mathcal{M}) = \langle \mathcal{O}, D(\mathcal{C}), D(\widetilde{S^2\mathcal{C}}),
        \ldots, D(\widetilde{S^{g-1}\mathcal{C}}) \rangle,
    \end{equation}
    where $\widetilde{S^k\mathcal{C}}$
    denotes the root stack obtained from
    $\mathbb{P}^k$ by extracting the square root from $2g+1$ hyperplanes
    in general position.
\end{conjecture}

Let us say a few words about the stacks $\widetilde{S^k\mathcal{C}}$ appearing
in the decomposition~\eqref{eq:dbstack}. Let $p_i^*$ denote
the~projective dual to $p_i$ in $(\mathbb{P}^1)^*\simeq \mathbb{P}^1$.
Let $h_i^*$ denote the image of $p^*_i$ under the degree $k$ Veronese
embedding $(\mathbb{P}^1)^*\hookrightarrow \mathbb{P}^k$, where we treat
the latter space as $(\mathbb{P}^k)^*$. Finally, let $H_i$ denote the
hyperplane in $\mathbb{P}^k$ dual to $h_i^*$. It is well known that the
points $h_1^*, h_2^*,\ldots, h_d^*$ are in general position, so the same holds for
the hyperplanes $H_1, H_2, \ldots, H_d$; that is, for any subset
$I\subset \{1, 2, \ldots d\}$ the intersection $\cap_{i\in I}H_i = \mathbb{P}^{k-|I|}$,
where if $|I|> k$, then the latter is empty. In particular, the collection
of effective Catrier divisors $H_1, H_2, \ldots, H_d$ forms
a \emph{generalized snc divisor:} for any subset $I\subset \{1, 2, \ldots d\}$
and $j\in I$ the inclusion $\cap_{i\in I}H_i \to \cap_{i\in I\setminus \{j\}}H_i\cap H_i$
is an effective Cartier divisor. In such a situation one can consider the
fibre product of the root stacks $\mathbb{P}^k_{2^{-1}H_i}$ over $\mathbb{P}^k$,
where $\mathbb{P}^k_{2^{-1}H_i}$ denotes the usual square root stack construction
associated with the Cartier divisor $H_i$. This fibre product
is precisely our stack $\widetilde{S^k\mathcal{C}}$.

\begin{remark}
Form~\cite[Theorem~4.9]{BLS} it follows that the derived category $D(\widetilde{S^k\mathcal{C}})$
carries a full exceptional collection. Thus, the same should be true for $D(\mathcal{M})$.
\end{remark}

\section{Computing the rank of \texorpdfstring{$K_0(\mathcal{M})$}{K0M}}
As evidece in support of our conjecture, we compute the~ranks of the Grothendieck
groups of the left- and right-hand side of~\eqref{eq:dbstack},
which we denote by $r_g$ and $l_g$, respectively.
Not only we show that they are equal, but give a very simple closed formula
for both.

\begin{proposition}
    There are equalities $l_g=r_g=g\cdot4^{g-1}$.
\end{proposition}

\begin{proof}

The derived category of $\widetilde{S^k\mathcal{C}}$
has a semiorthogonal decomposition indexed by subsets
whose components are the derived categories of the intersections
$\cap_{i\in I} H_i=\mathbb{P}^{k-|I|}$, where the index $I$ runs over the~subsets
$I\subseteq \{1, \ldots, 2g+1\}$
and $H_i$ is the~$i$-th hypersurface (see~\cite[Theorem~4.9]{BLS}).
We conclude that the rank of $K_0$ of the right-hand side of~\eqref{eq:dbstack}
equals
\begin{equation}\label{eq:rg}
    r_g = \sum_{k=0}^{g-1}\sum_{t=0}^k (k+1-t){\binom{2g+1}{t}}
    =\sum_{t=0}^{g-1}{\binom{2g+1}{t}}\sum_{k=t}^{g-1}(k+1-t)
    =\sum_{t=0}^{g-1}{\binom{g+1-t}{2}}{\binom{2g+1}{t}}.
\end{equation}

Next, we use the interpretation of $\mathcal{M}$ as the moduli space of parabolic
bundles. A series of varieties $\mathcal{M}_0, \mathcal{M}_1,\ldots \mathcal{M}_{g-1}$
was constructed in~\cite{Ba} such that
$\mathcal{M}_0=\mathbb{P}^{2g-2}$,
$\mathcal{M}_1$ is a blowup of $\mathcal{M}_0$ in $(2g+1)$ points,
$\mathcal{M}_{g-1}\simeq\mathcal{M}$, and $\mathcal{M}_{i+1}$
can be obtained from $\mathcal{M}_i$ via an anti-flip:
in $\mathcal{M}_i$ one should blow up
\begin{equation*}
    n_i={\binom{2g+1}{i+1}}+{\binom{2g+1}{i-1}}+
    {\binom{2g+1}{i-3}}+\cdots
\end{equation*}
subvarieties isomorphic to $\mathbb{P}^i$,
and then blow down the exceptional divisors to subvarieties isomorphic to
$\mathbb{P}^{2g-3-i}$ in $\mathcal{M}_{i+1}$.
We use this description to inductively compute $l_g = \rk\mathcal{M}=\rk\mathcal{M}_{g-1}$.

The rank of $K_0(\mathcal{M}_0)=2g-1$. When we pass from $\mathcal{M}_i$ to $\mathcal{M}_{i+1}$,
each blow up increases the rank of $K_0$ by
\begin{equation*}
    (2g-3-i)\rk K_0(\mathbb{P}^i) = (2g-3-i)(i+1),
\end{equation*}
while each blow down decreases the rank of $K_0$ by
\begin{equation*}
    i\cdot\rk K_0(\mathbb{P}^{2g-3-i}) = i(2g-2-i).
\end{equation*}
(Formally, there is no blow down when passing from $\mathcal{M}_0$ to $\mathcal{M}_1$,
but the formula still applies.)
Since there are $n_i$ of each, we conclude that
\begin{equation*}
    \rk K_0(\mathcal{M}_{i+1})-\rk K_0(\mathcal{M}_i) = n_i(2g-3-2i).
\end{equation*}
Collecting the coefficients at ${\binom{2g+1}{i}}$, we conclude that $l_g$ is given by
the same sum as $r_g$ in~\eqref{eq:rg}.

Finally, we show that $r_g=g\cdot 4^{g-1}$. Using the binomial recurrence relation,
rewrite the right-hand side of~\eqref{eq:rg} as
\begin{equation*}
    r_g = \sum_{t=0}^{g-1} (g-t)^2\binom{2g}{t} = \frac{1}{2}\sum_{t=0}^{2g}(g-t)^2\binom{2g}{t}.
\end{equation*}

Consider a binomially distributed random variable $X\sim B(n, p)$. It is well known that
its expected value $\mathrm{E}[X]$ equals $np$. Meanwhile, its variance $\mathrm{Var}(X)=np(1-p)$.
Let us specialize to $n=2g$ and $p=\frac{1}{2}$. Then $\mathrm{E}[X]=g$, and
\begin{equation*}
    \frac{g}{2} = \mathrm{Var}(X) = \sum_{t=0}^{2g}(g-t)^2\binom{2g}{t}\left(\frac{1}{2}\right)^{2g} = 
    \left(\frac{1}{2}\right)^{2g-1}\frac{1}{2}\sum_{t=0}^{2g}(g-t)^2\binom{2g}{t} =
    \left(\frac{1}{2}\right)^{2g-1} r_g,
\end{equation*}
from which we get the desired equality $r_g = g\cdot 4^{g-1}$.

\end{proof}

\begin{remark}
    It would be very interesing to find a combinatorial proof of the identity
    \begin{equation}\label{eq:main}
        \sum_{t=0}^{g-1}{\binom{g+1-t}{2}}{\binom{2g+1}{t}} = g\cdot 4^{g-1}.
    \end{equation}
    The right-hand side has an elementary interpretation: it counts pairs of
    subsets $A, B\subset \{1, 2, \ldots, g\}$ with non-empty intersection and
    a choice of an element $x\in A\cap B$.
\end{remark}

\noindent\textbf{Acknowledgements.} The author is grateful to A.~Kuznetsov
and P.~Belmans for interesting conversations. Our original proof of~\eqref{eq:main}
was much less elegant. We would like to thank Alapan Das
for a basic probability course refresher on MathOverflow.

\printbibliography

\end{document}